\documentclass[12pt,twoside]{amsart}

\usepackage{amsmath,amssymb}
\usepackage{amsthm}
\usepackage{ascmac,color,xcolor,braket}
\usepackage[pdftex]{graphicx}
\usepackage[abbrev]{amsrefs}
\usepackage{latexsym}

\newcommand{\supp}{\mathop{\rm supp}\nolimits}

\newcommand{\cd}{\mathop{\rm CD}\nolimits}

\newcommand{\dr}{\mathop{\rm DR}\nolimits}

\newcommand{\setrel}{\mathrel{}\middle|\mathrel{}}
\newcommand{\id}{\text{id}}
\newcommand{\pr}{\text{pr}}
\newcommand{\mf}{\mathfrak}
\newcommand{\cl}{\mathcal}

\newtheorem{thm}{Theorem}[section]
\newtheorem{prop}[thm]{Proposition}
\newtheorem{lem}[thm]{Lemma}
\newtheorem{cor}[thm]{Corollary}
\newtheorem{remark}[thm]{Remark}
\newtheorem{example}[thm]{Example}
\theoremstyle{definition} 
\newtheorem{defn}[thm]{Definition}
\newtheorem*{assump*}{Assumption A}
\newtheorem*{ack*}{Acknowledgment}

\makeatletter

\@addtoreset{equation}{section}
\makeatother

\begin{document}

\title{
{$L^1$-}Monge problem in metric spaces possibly with branching geodesics
}
\author{Shinichiro Kobayashi}
\address{Mathematical Institute, Tohoku University, Sendai 980-8578, Japan}
\email{shin-ichiro.kobayashi.p3@dc.tohoku.ac.jp}
\date{August 2, 2019.} 
\maketitle

\begin{abstract}
In this paper, we consider the Monge optimal transport problem with distance cost.
We prove that in some metric spaces, possibly with many branching geodesics, an optimal transport map exists if the first marginal is absolutely continuous.
The result is applicable to normed spaces and Hilbert geometries.
\end{abstract}

\section{Introduction}
Mass transportation problem has its origin in Monge-Kantorovich \cite{Monge1781, Kantorovich} and is stated as follows.
Let $X$ be a topological space and let $\mu_1$ and $\mu_2$ be two Borel probability measures on $X$.
Let $c\colon X^2\to [0,+\infty]$ be a function and call it a \emph{cost function}.
The Monge mass transportation problem, the Monge problem for short, with the cost function $c$ asks whether the infimum
\begin{equation}\label{Monge}
\inf_{T}\int_{X}c(x,T(x))\, d\mu_1(x) \tag{M}
\end{equation}
is achieved, where $T$ runs over all Borel measurable maps from $X$ to $X$ satisfying $T_{\sharp}\mu_1=\mu_2$.
We call such a map $T$ a \emph{transport map} from $\mu_1$ to $\mu_2$.
The measure $T_{\sharp}\mu_1$ is the push-forward measure of $\mu_1$ under $T$.
We call the objective functional in (\ref{Monge}) the \emph{total cost functional}.
A minimizer in (\ref{Monge}) is called a $c$-\emph{optimal transport map} from $\mu_1$ to $\mu_2$.
The lack of linearity and coercivity of the total cost functional makes the Monge problem difficult.
Kantorovich formulated a relaxation of the problem and overcame these difficulties.
His idea is to use so-called \emph{transport plans}, also known as joint distributions, instead of maps.
The set of transport plans, $\Pi(\mu_1,\mu_2)$, is defined by
\[
\Pi(\mu_1,\mu_2):=\left\{ \pi\in\cl{P}({X^2}) \setrel (\pr_i)_{\sharp}\pi=\mu_i, i=1,2\right\},
\]
where the map $\pr_i\colon X^2\to X$ is the canonical projection for $i=1,2$ and the set $\cl{P}(X^2)$ denotes the set of all Borel probability measures on $X^2$.
Kantorovich proposed to study the attainability of the infimum
\begin{equation}\label{Kantorovich}
\inf_{\pi\in\Pi(\mu_1,\mu_2)}\int_{X^2}c(x,y)\, d\pi(x,y). \tag{K}
\end{equation}
A minimizer in (\ref{Kantorovich}) is called a $c$-\emph{optimal transport plan} from $\mu_1$ to $\mu_2$.
If $X$ is a Polish space and if the cost function $c$ is lower semi-continuous, then an optimal transport plan  exists.
In order to obtain an optimal transport map, it is sufficient to see that a certain optimal transport plan is induced by a map.

Let  $(X,d)$ be a complete separable metric space. We focus on the case that the cost function is the distance function $d$.
Under a certain lower bound condition of curvature, the existence of an optimal transport map is proved by many researchers.
For metric spaces without non-branching geodesics, see \cite{BiCa11, BiCa13, Cava17}.
For normed spaces, see e.g., \cite{AKP04, AP03, CP11}.

In this paper, inspired by the work of Champion-Pascale \cite{CP11},
we consider the Monge problem on a metric space $(\Omega,\rho)$, where $\Omega$ is a bounded convex $G_{\delta}$-set in an Euclidean space with non-empty interior and $\rho$ is a metric on $\Omega$ satisfying some projectivity conditions.
\begin{assump*}
We assume the following three conditions on a metric $\rho$ on $\Omega$.
\begin{enumerate}
\renewcommand{\labelenumi}{{\rm(\roman{enumi})}}
\item The topology induced by $\rho$ coincides with  the Euclidean one.
\item Any line segment is a geodesic in $(\Omega,\rho)$.
\item The $n$-dimensional Lebesgue measure $\cl{L}^n$ on $(\Omega,\rho)$ is locally doubling.
\end{enumerate}
\end{assump*}
Note that we do not assume that $(\Omega,\rho)$ is non-branching. Under these settings, we have the following main theorem.
\begin{thm}\label{thm: main theorem}
Let $\Omega$ be a bounded convex $G_{\delta}$-set in the $n$-dimensional Euclidean space $\mathbb{R}^n$
with non-empty interior and $\rho$ be a metric on $\Omega$ satisfying {\rm Assumption A}. Let $\mu_1$ and $\mu_2$ be two Borel probability measures on $\Omega$ with compact support.
If $\mu_1$ is absolutely continuous with respect to the $n$-dimensional Lebesgue measure $\cl{L}^n$, then there exists a $\rho$-optimal transport map from $\mu_1$ to $\mu_2$.
\end{thm}
The idea of the proof of Theorem \ref{thm: main theorem} is based on \cite{CP11}.
We modify a variational approximation of transport plans, introduced in \cite{CP11}, and use it.
To see that our approximation scheme works well, we need an appropriate grid for a projective metric.
Unlike the case of normed spaces, it is not clear that there exists such a grid.
If we assume (iii), then this difficulty is resolved.
In this case, the doubling dimension may be different from $n$.
Note that our approximation scheme depends on the doubling dimension.

The Hilbert metric on an bounded convex open set in an Euclidean space satisfies Assumption A.
The Hilbert metric is a generalization of the Cayley-Klein model of the hyperbolic geometry.
For a bounded convex open set $\Omega$ in an Euclidean space endowed with the Hilbert metric is called the Hilbert geometry for $\Omega$. 
Under some regularity assumptions on the boundary of a domain, the Hilbert geometry is regarded as a Finsler manifold of constant flag curvature $-1$.
We obtain the existence result of an optimal transport map for a Hilbert geometry as a corollary. 
\begin{cor}\label{cor: Hilb geom}
Let $\Omega\subset \mathbb{R}^n$ be a bounded convex open set and $h_{\Omega}$ the Hilbert metric on $\Omega$.
Let $\mu_1$ and $\mu_2$ be two Borel probability measures on $\Omega$ with compact support.
If $\mu_1$ is absolutely continuous with respect to the $n$-dimensional Lebesgue measure $\cl{L}^n$, then there exists an $h_{\Omega}$-optimal transport map from $\mu_1$ to $\mu_2$.
\end{cor}
\begin{ack*}
The author would like to thank Professor Takashi Shioya for helpful comments.
\end{ack*}

\section{Preliminaries}

\subsection{General facts from optimal transport theory}
In this section, we recall some basics on optimal transport theory.
We refer to \cite{AG13, Vil09} for more details.
For a topological space $X$, denote by $\cl{P}(X)$ the set of all Borel probability measures on $X$.
\begin{defn}
For two topological spaces $X$ and $Y$, a Borel measurable map $f\colon X\to Y$ and a Borel probability measure $\mu\in\cl{P}(X)$, 
we define the \emph{push-forward measure} $f_{\sharp}\mu$ of $\mu$ under $f$ by
\[
f_{\sharp}\mu(B):=\mu\left(f^{-1}(B)\right)
\]
for any Borel set $B$ in $Y$.
\end{defn}
Let $\mu_1, \mu_2\in \mathcal{P}(X)$ be two Borel probability measures.
For a Borel measurable map $T\colon X\to X$, we say that $T$ is a \emph{transport map} from $\mu_1$ to $\mu_2$ if
$T_{\sharp}\mu_1=\mu_2$ holds. The set of all transport maps from $\mu_1$ to $\mu_2$ is denoted by $\mathcal{M}(\mu_1,\mu_2)$.
For a Borel probability measure $\pi\in \mathcal{P}(X^2)$, we say that $\pi$ is a \emph{transport plan} from $\mu_1$ to $\mu_2$ if
$(\pr_i)_{\sharp}\pi=\mu_i$ for $i=1,2$ holds, where the map $\pr_i\colon X^2\to X$ is the canonical projection for $i=1,2$.
In this case, we say that $\mu_1$ is \emph{the first marginal of $\pi$} and that $\mu_2$ is \emph{the second marginal of $\pi$}.
The set of all transport plans from $\mu_1$ to $\mu_2$ is denoted by $\Pi(\mu_1,\mu_2)$.
For a transport map $T\in\cl{M}(\mu_1,\mu_2)$, we denote $(\id, T)_{\sharp}\mu_1$ by $\pi_T$, where $(\id, T)\colon X\to X^2$ is the map that assigns each $x\in X$ to $(x, T(x))$.

\begin{defn}
Let $c\colon X^2\to [0,+\infty]$ be a function, which we call a \emph{cost function}. A transport map $T_0\in \mathcal{M}(\mu,\mu_2)$ is a \emph{$c$-optimal transport map} if it minimizes 
\[
I_{\mathrm{M}}(T):=\int_{X}c(x,T(x))\, d\mu(x),\, T\in\mathcal{M}(\mu,\mu_2).
\]
A transport plan $\pi_0\in\Pi(\mu_1,\mu_2)$ is a \emph{$c$-optimal transport plan} if it minimizes 
\[
I_{\mathrm{K}}(\pi):= \int_{X^2}c(x,y)\, d\pi(x,y),\, \pi\in \Pi(\mu_1,\mu_2).
\]
\end{defn}
The minimization problem for $I_M$ and $I_K$ is called the Monge problem and the Kantorovich problem for the cost function $c$, respectively.
By a direct argument, the existence of a $c$-optimal transport plan is ensured if $X$ is a Polish space and if $c$ is lower semi-continuous.
Kantorovich problem admits a dual formulation. We recall the notion of $c$-transform and $c$-concavity.
\begin{defn}
Let $\varphi\colon X\to [-\infty,+\infty]$ be a function.
\begin{enumerate}
\item The \emph{$c$-transform} $\varphi^c$ of $\varphi$ is defined by
\[
\varphi^c(y):= \inf_{x\in X}(c(x,y)-\varphi(x)),\, y\in X.
\]
\item $\varphi$ is \emph{$c$-concave} if there exists a function $\psi\colon X\to[-\infty,+\infty]$ on $X$ such that $\varphi=\psi^c$ holds.
\end{enumerate}
\end{defn}

\begin{example}
Let $(X,d)$ be a metric space and $\varphi\colon X\to\mathbb{R}$ a function.
Then, $\varphi$ is $d$-concave if and only if $\varphi$ is $1$-Lipschitz continuous.
\end{example}

\begin{defn}
A subset $\Gamma\subset X^2$ is said to be \emph{c-cyclically monotone} if
for any finitely many  points $(x_1,y_1),\ldots,(x_n,y_n)\in \Gamma$, we have
\[
\sum_{i=1}^n c(x_i,y_i)\leq \sum_{i=1}^n c(x_i,y_{i+1}),
\]
where $y_{n+1}:=y_1$. A Borel probability measure $\pi$ on $X^2$ is said to be \emph{$c$-cyclically monotone}
if $\pi$ is concentrated on a $c$-cyclically monotone set.
\end{defn}

In the case that $(X,d)$ is the $n$-dimensional Euclidean space $\mathbb{R}^n$ and $c$ is the square of the Euclidean metric, the $c$-cyclical monotonicity of $\Gamma$ yields the monotonicity of $\Gamma$, i.e., for any $(x,y), (x',y')\in \Gamma$, we have
\[
(y-y')\cdot (x-x') \geq 0,
\]
where the dot is the Euclidean inner product of $\mathbb{R}^n$.
Next we recall the notion of Kantorovich potential.
Denote by $\Phi_c$ the set of all pairs of functions $(\varphi,\psi)\in L^1(X,\mu_1)\times L^1(X,\mu_2)$ such that
\[
\varphi(x)+\psi(y)\leq c(x,y)
\]
holds for $\mu_1$-almost all $x\in X$ and $\mu_2$-almost all $y\in X$.
Denote by $\Psi_c$ the set of all pairs of functions $(\varphi,\psi)\in C_{b}(X)\times C_{b}(X)$ such that
\[
\varphi(x)+\psi(y)\leq c(x,y)
\]
holds for all $x, y\in X$.
For $(\varphi,\psi)\in\Phi_c$, we define $J(\varphi,\psi)$ by
\[
J(\varphi,\psi):=\int_{X}\varphi\, d\mu_1 + \int_{X}\psi\, d\mu_2.
\]
We call every $(\varphi,\psi)\in\Phi_c$ maximizing  $J$ a \emph{Kantorovich potential}.
The following theorem is known as the Kantorovich duality.
\begin{thm}[Kantorovich duality]\label{KRduality}
Let $c\colon X^2 \to [0,+\infty]$ be a lower semi-continuous function. Then we have the equalities
\[
\min_{\pi\in\Pi(\mu_1,\mu_2)}I_{\mathrm{K}}(\pi)
=\sup_{(\varphi,\psi)\in \Phi_c}J(\varphi,\psi)
=\sup_{(\varphi,\psi)\in \Psi_c}J(\varphi,\psi).
\]
\end{thm}

\begin{thm}
Let $c\colon X^2\to [0,+\infty]$ be a lower semi-continuous function.
If $\pi\in\Pi(\mu_1,\mu_2)$ is optimal and if $I_{\mathrm{K}}(\pi)$ is finite, then $\pi$ is concentrated on a $c$-cyclically monotone set.
Moreover, if $c$ is real-valued, then there exists a $c$-cyclically monotone set $\Gamma\subset X^2$ such that for any $\pi\in\Pi(\mu_1,\mu_2)$, the following are equivalent to each other.
\begin{enumerate}
\item $\pi$ is a $c$-optimal transport plan from $\mu_1$ to $\mu_2$.
\item $\pi$ is $c$-cyclically monotone.
\item There exists a $c$-concave Borel measurable function $\varphi\colon X\to \mathbb{R}$ such that 
for $\pi$-almost every point \ $(x,y)\in X\times Y$, we have
\[
\varphi(x)+\varphi^c(y)=c(x,y).
\]
\item $\pi$ is concentrated on $\Gamma$.
\end{enumerate}
\end{thm}

\begin{thm}\label{thm: potential}
Let $c\colon X^2\to [0,+\infty)$ be a lower semi-continuous function. 
Assume that $\cl{I}_c(\mu_1,\mu_2)<+\infty$,
\[
\mu_1\left(\left\{ x\in X \setrel \int_Y c(x,y)\, d\mu_2(y) < +\infty\right\}\right) > 0,
\]
and
\[
\mu_2\left(\left\{ y\in X \setrel \int_X c(x,y)\, d\mu_1(x) < +\infty\right\}\right) > 0.
\]
Then, there exists a $c$-concave function $\varphi\colon X \to \mathbb{R}$ such that the pair $(\varphi, \varphi^c)$ is a Kantorovich potential.
\end{thm}

Let $(X,d)$ be a complete and separable metric space and let $p\geq 1$.
We say that a measure $\mu\in \mathcal{P}(X)$ has \emph{finite $p$-th moment with respect to $d$} if there exists a point $x_0\in X$ such that
\[
\int_X d^p(x,x_0)\, d\mu(x)<+\infty.
\]
Denote by $\mathcal{P}_p(X,d)$ the set of all probability measures with finite $p$-th moment with respect to $d$.
Note that for $\mu_1,\mu_2\in \mathcal{P}_p(X)$, the assumptions in Theorem \ref{thm: potential} are fulfilled for $c=d^p$.
\emph{The $L^p$-Wasserstein metric} $W_p$ on $\mathcal{P}(X)$ is defined by
\[
W_p(\mu_1,\mu_2):=\inf\left\{\left(\int_{X^2}d^p(x,y)\, d\pi(x,y)\right)^{1/p} \setrel \pi\in\Pi(\mu_1,\mu_2)\right\},
\]
i.e., $W_p(\mu_1,\mu_2)$ is the $p$-th root of the $d^p$-optimal transport cost from $\mu_1$ to $\mu_2$. The function $W_p$ is indeed a metric on $\mathcal{P}_p(X,d)$.

\subsection{Some facts on metric spaces}
In this section, we enumerate some facts and prove some claims on metric spaces.

Let $(X,d)$ be a metric space. 
We call a discrete subset of $X$ a \emph{net}.
For a net $\mathcal{N}\subset X$ and $\varepsilon>0$,
we say that $\mathcal{N}$ is an $\varepsilon$-\emph{net} if the $\varepsilon$-neighborhood $B_{\varepsilon}(\mathcal{N})$ of $\mathcal{N}$
coincides with the whole $X$.
We say that $\mathcal{N}$ is $\varepsilon$-\emph{discrete} if $d(x,y)>\varepsilon$ for any two distinct points $x,y\in \mathcal{N}$.
The set of all $\varepsilon$-discrete nets equipped with the inclusion relation forms a poset and we see that  any maximal element of the set is an $\varepsilon$-net.
\begin{defn}
Let $\mathfrak{m}$ be a non-negative Borel measure on $X$.
We say that $\mathfrak{m}$ is \emph{locally doubling} if for any $R>0$, there exists a constant $C=C_R\geq 1$ such that
$\mathfrak{m}(B(x,2r))\leq C\mathfrak{m}(B(x,r))$ for any $x\in X$ and $0<r\leq R$.
\end{defn}
The locally doubling condition is equivalent to
\begin{equation}\label{ineq: doubling}
\mathfrak{m}(B(x,r'))\leq C^{\lceil \log_2(r'/r) \rceil}\mathfrak{m}(B(x,r))
\end{equation}
for any $x\in X$ and $0<r<r'\leq R$,
where $\log_2$ is the base $2$ logarithm function and $\lceil\cdot\rceil$ is the ceil function.
\begin{lem}\label{lem: numberofnets}
Let $X$ be a metric space, let $B\subset X$ be a bounded Borel set with diameter $D$ and let $\mathcal{N}\subset B$ be an $\varepsilon$-discrete net of $B$. Assume that $X$ admits a locally doubling Borel measure $\mf{m}$. Then the number $\#\mathcal{N}$ of $\mathcal{N}$ is bounded above by $C^{\lceil \log_2(2D/\varepsilon)\rceil}$.
\end{lem}

\begin{proof}
The $\varepsilon$-discreteness of $\cl{N}$ yields that a family of $\varepsilon/2$-balls \\
$\{B(x,\varepsilon/2) \mid x\in \cl{N}\}$ is pairwise disjoint. Combining this fact with the doubling condition (\ref{ineq: doubling}) for $R=D$ and $r=\varepsilon/2$, we obtain
\[
\mf{m}(B)\geq \mf{m}\left(\bigcup_{x\in \cl{N}}B(x,\varepsilon/2)\right)
 = \sum_{x\in \cl{N}} \mf{m}(B(x,\varepsilon/2))
\geq \frac{\#\cl{N}\mf{m}(B)}{C^{\lceil \log_2 (2D/\varepsilon)\rceil}},
\]
which yields the lemma.
\end{proof}

\begin{defn}
Let $A\subset X$ be any subset of $X$. A map $p\colon X\to A$ is called a \emph{nearest point projection to} $A$
if $d(x,p(x)) = d(x,A)$ for any $x\in X$, where $d(x,A)=\inf_{a\in A} d(x,a)$.
\end{defn}

\begin{lem}
If $A=\{a_1,\ldots,a_n\}\subset X$ is a finite subset, then there exists a Borel measurable nearest point projection to $A$.
\end{lem}

\begin{proof}
We put
\[
i(x):=\min\{i\mid d(x,a_i)\leq d(x,a_j)\ \text{for any}\ j\}
\]
and $p(x):=a_{i(x)}$. The assertion follows from
\[
p^{-1}(a_i)=\left(\bigcap_{j=1}^n\{f_i\leq f_j\}\right)\cap \left(\bigcap_{j<i}\bigcup_{k=1}^n\{f_j > f_k\}\right), i=1,\ldots,n,
\]
where $f_i$ is the distance function from $a_i$.
\end{proof}

\section{Proof of Main Theorem}
In order to prove the existence of an optimal transport map, it suffices to show that a certain optimal transport plan is induced by a transport map.
The following lemma gives a criterion according to which a measure in a product space is induced by a map.
\begin{lem}\label{lem: map}
Let $X$ and $Y$ be two Polish spaces and let $\pi\in \cl{P}(X\times Y)$.
Let $\mu_1$ be the first marginal of $\pi$.
Then, $\pi$ is induced by a map if and only if
there exists a Borel measurable set $\Gamma\subset X\times Y$ of full $\pi$-measure such that
for $\mu_1$-almost every $x\in X$ there exists a unique $y=T(x)\in Y$ with $(x,y)\in \Gamma$.
In this case, the map $T$ is $\mu_1$-measurable and $\pi=({\rm id}, T)_{\sharp}\mu_1$.
\end{lem}

Throughout this section, let $\Omega\subset\mathbb{R}^n$ be a bounded and convex $G_{\delta}$-set with non-empty interior and let $\rho$ be a metric on $\Omega$ satisfying Assumption A (see Section 1).
The idea that we treat $G_{\delta}$-sets is based on Mazurkiewictz's theorem (for the proof, see e.g.\cite{Wil70}).
\begin{thm}[Mazurkiewictz]
Let $X$ be a completely metrizable topological space and $A$ a subspace of $X$.
Then $A$ is completely metrizable if and only if $A$ is a $G_{\delta}$-set.
\end{thm} 
Let $\lambda$ be a Borel probability measure on $\Omega$ that is absolutely continuous with respect to the $n$-dimensional Lebesgue measure $\cl{L}^n$. 
We call the function $f$ defined by
\[
f(x):= \limsup_{r\to 0+}\frac{\lambda(B(x,r))}{\mf{m}(B(x,r))}
\]
the \emph{density} of $\lambda$
and denote it by $d\lambda/d\cl{L}^n$.

For two probability measures $\mu_1,\mu_2\in \cl{P}(\Omega)$, let $\mathcal{O}_{0}(\mu_1,\mu_2)$ be the set of all $\rho$-optimal transport plans from $\mu_1$ to $\mu_2$.
We construct two families $\cl{O}_1(\mu_1,\mu_2)$ and $\mathcal{O}_2(\mu_1,\mu_2)$ of optimal transport plans and show that each of elements in $\mathcal{O}_2(\mu_1,\mu_2)$ is induced by a map. As a result, the set $\mathcal{O}_2(\mu_1,\mu_2)$ consists of a single element.
The construction of $\cl{O}_2(\mu_1,\mu_2)$ is similar to that of \cite{San09} and \cite{CP11}.
We first define the set $\cl{O}_1(\mu_1,\mu_2)$. 
We consider the secondary variational problem:
\[ \label{svp1}
\inf \left\{ \int_{\Omega^2}\lvert x-y \rvert^2\, d\pi \setrel \pi \in \cl{O}_0(\mu_1,\mu_2)\right\}. \tag{SVP1} 
\]
We denote by $\cl{O}_1(\mu_1,\mu_2)$ the set of all solutions of (\ref{svp1}).
Since the topology of $(\Omega,\rho)$ is Euclidean, it is easy to see that $\cl{O}_0(\mu_1,\mu_2)$ is non-empty.
On the other hand, we cannot deduce the monotonicity of a measure in $\cl{O}_1(\mu_1,\mu_2)$ due to the lack of elements in the feasible region in (\ref{svp1}). Nevertheless, we will see that a weak version of monotonicity holds for any measure in $\cl{O}_1(\mu_1,\mu_2)$.
We say that a subset $\Gamma$ of $\Omega^2$ is \emph{restrictedly monotone} if for any $(x,y), (x',y')\in \Gamma$ with $x\in [x',y']$, we have
\[
(y-y')\cdot (x-x')\geq 0,
\]
where the set $[x',y']\subset \Omega$ is the line segment between $x'$ and $y'$.
A Borel measure $\pi$ on $\Omega^2$ is said to be \emph{restrictedly monotone} if $\pi$ is concentrated on a restrictedly monotone subset of $\Omega^2$.
We will see that any measure in $\cl{O}_1(\mu_1,\mu_2)$ is restrictedly monotone.
We replace the problem (\ref{svp1}) with an equivalent one whose feasible region is the whole $\Pi(\mu_1,\mu_2)$.
Accordingly, we must change the objective functional. We fix a Kantorovich potential $\varphi$ for $W_1(\mu_1,\mu_2)$ and
put 
\[
\beta(x,y) =
\begin{cases}
\lvert x-y\rvert^2 & \varphi(x)-\varphi(y)=\rho(x,y), \\
+\infty & otherwise.
\end{cases}
\]
The ordinary optimal transport problem with cost function $\beta$
\[ \label{svp2}
\inf \left\{ \int_{\Omega^2}\beta(x,y)\, d\pi \setrel \pi \in \Pi(\mu_1,\mu_2)\right\} \tag{SVP2}
\]
is equivalent to the secondary variational problem (\ref{svp1}). More precisely, we have the following lemma.
\begin{lem}\label{equivalence}
Let $\mu_1, \mu_2, \varphi$ and $\beta$ be as above.
Then, for any $\pi\in\cl{O}_0(\mu_1,\mu_2)$, $\pi$ is a solution of {\rm (SVP1)} if and only if it is a solution of {\rm (SVP2)}.
\end{lem}
By virtue of Lemma \ref{equivalence}, we see that for any $\pi\in \cl{O}_1(\mu_1,\mu_2)$,
there exists a $\beta$-cyclically monotone set $\Gamma$ on which $\pi$ concentrates.
We also observe that for any $\pi\in \cl{O}_{0}(\mu_1,\mu_2)$, $\beta$ takes finite values on $\Gamma$.
Furthermore, the $\beta$-cyclical monotonicity is a weak version of the monotonicity.
More generally, we have the following. In the proof, we use the condition on  geodesics for $\rho$. 
\begin{prop}\label{lem: monotone}
Let $\varphi\colon \Omega\to\mathbb{R}$ be a $1$-Lipschitz continuous function.
Define a extended-real valued function $\beta\colon \Omega^2\to [0,+\infty]$ by
\[
\beta(x,y):=
\begin{cases}
\lvert x-y \rvert^2 & {\rm if}\ \varphi(x)-\varphi(y)=\rho(x,y), \\
+\infty & otherwise.
\end{cases}
\]
Let $\Gamma\subset \Omega^2$ be a $\beta$-cyclically monotone set on which $\beta$ is finite.
Then $\Gamma$ is restrictedly monotone.
\end{prop}
\begin{proof}
By the $\beta$-cyclical monotonicity of $\Gamma$ and the finiteness of $\beta$, we have
\[
\lvert x-y \rvert^2 + \lvert x'- y' \rvert^2 \leq \beta(x,y') + \beta(x',y).
\]
Due to the $1$-Lipschitz continuity of $\varphi$, it suffices to prove that 
\begin{equation}\label{lem: monotonicity ineq 1}
\varphi(x)-\varphi(y') \geq \rho(x,y')
\end{equation}
and 
\begin{equation}\label{lem: monotonicity ineq 2}
\varphi(x')-\varphi(y) \geq \rho(x',y).
\end{equation}
Note that the argument is not symmetric.
We first prove (\ref{lem: monotonicity ineq 2}).
Since $x', x$ and $y$ are aligned, we have
\begin{equation}\label{eq: align}
\rho(x', y') = \rho(x', x) + \rho(x, y').
\end{equation}
Combining (\ref{eq: align}) with the $1$-Lipschitz continuity of $\varphi$ yields
\begin{align*}
\varphi(x')-\varphi(x) &\leq \rho(x,x') \\
                             &= \rho(x',y') - \rho(x,y') \\
                             &\leq \varphi(x')-\varphi(y')-(\varphi(x)-\varphi(y')) \\
                             &= \varphi(x')-\varphi(x).
\end{align*}
In particular, we have
\begin{equation}\label{eq: aaa}
\varphi(x') - \varphi(x) = \rho(x',x).
\end{equation}
Combining (\ref{eq: aaa}) with the triangle inequality, we arrive at
\begin{align*}
\rho(x',y) &\leq \rho(x',x) + \rho(x,y) \\
                      &= \varphi(x')-\varphi(x)+\varphi(x)-\varphi(y) \\
                      &= \varphi(x')-\varphi(y),
\end{align*}
which proves (\ref{lem: monotonicity ineq 2}).
We next prove (\ref{lem: monotonicity ineq 1}). By using the triangle inequality for $x', x$ and $y$, we obtain
\begin{align*}
\rho(x,y') &\leq \rho(x,y')+\rho(x',x)+\rho(x,y)-\rho(x',y)  \\
                      &= \rho(x,y)+\rho(x',y')-\rho(x',y) \\
                      &= \varphi(x)-\varphi(y)+\varphi(x')-\varphi(y')-(\varphi(x')-\varphi(y)) \\
                      &= \varphi(x)-\varphi(y').
\end{align*}
The proof is completed.
\end{proof}

Next, we construct a subset $\cl{O}_2(\mu_1,\mu_2)$ of $\cl{O}_1(\mu_1,\mu_2)$, which plays a crucial role in the proof of the main theorem.
We assume that the support of $\mu_1$ and $\mu_2$ is compact, respectively.
We denote by $\cl{P}_c(\Omega)$ the set of all compactly supported Borel probability measures on $\Omega$. 
Let $\varepsilon>0$ be a positive real number. We define a functional $C_{\varepsilon}\colon \cl{P}(\Omega^2)\to [0,+\infty]$ by
\[
C_{\varepsilon}(\pi):=\frac{1}{\varepsilon}W_1((\pr_2)_{\sharp}\pi,\mu_2) + \int_{\Omega^2}c_{\varepsilon}\, d\pi+\varepsilon^{3d+2}\#\supp((\pr_2)_{\sharp}\pi),
\]
where $d$ is the the base 2 logarithm of the doubling constant of $\supp\mu_2$ and $c_{\varepsilon}(x,y):=\rho(x,y)+\varepsilon\lvert x-y\rvert^2$, $x,y\in \Omega$.
The functional $C_{\varepsilon}$ is lower semi-continuous with respect to the weak convergence of measures.
We denote by $D_{\varepsilon}$ the set of all minimizers of $C_{\varepsilon}$ whose first marginal is $\mu_1$:
\[
D_{\varepsilon}:=\arg\min\left\{ C_{\varepsilon}(\pi) \setrel \pi\in \cl{P}(\Omega^2), (\pr_1)_{\sharp}\pi=\mu_1 \right\}.
\]
We observe that the set $D_{\varepsilon}$ is not empty.
$\cl{O}_2(\mu_1,\mu_2)\subset\Pi(\mu_1,\mu_2)$ is defined to be the set of all cluster points of any sequence $\{\pi_{\varepsilon}\in D_{\varepsilon}\}_{\varepsilon>0}$.
By the compactness of the support of $\mu_1$ and $\mu_2$, we see that $\cl{O}_2(\mu_1,\mu_2)$ is not empty.
Furthermore, we see that any measure in $\cl{O}_2(\mu_1,\mu_2)$ solves (\ref{svp1}).
\begin{lem}\label{well-definedness of O2}
For any $\varepsilon>0$, let $\pi_{\varepsilon}\in D_{\varepsilon}$.
Then the second marginal of  $\pi_{\varepsilon}$ converges to $\mu_2$ weakly as $\varepsilon\to 0$.
Moreover, every limit point $\pi_*$ of $\{\pi_{\varepsilon}\}_{\varepsilon>0}$ is an element of $\cl{O}_1(\mu_1, \mu_2)$. 
\end{lem}

\begin{proof}
For every natural number $j$, we take a maximal $(1/j)$-net $\cl{N}_j$ of  $\supp\mu_2$.
Put $d:=\log_2 C$. By Lemma \ref{lem: numberofnets}, we have
$\#\cl{N}_j\leq O(j^d)$ for any $j$. Since $\cl{N}_j$ is finite, there exists a Borel measurable nearest point projection
$p_j\colon \supp\mu_2\to\cl{N}_j$.
Pick and fix a transport plan $\pi\in \cl{O}_1(\mu_1,\mu_2)$. Set $\pi'_j:=(\id\times p_j)_{\sharp}\pi$, where the map $\id\times p_j$ is defined by
$(\id\times p_j)(x,x'):=(x,p_j(x'))$. The measure $\pi'_j$ is a transport plan from $\mu_1$ to $\mu_2$.
Since $\pi_{\varepsilon}\in D_{\varepsilon}$, we have
\begin{align}
C_{\varepsilon}(\pi_{\varepsilon}) 
&= \frac{1}{\varepsilon}W_1((\pr_2)_{\sharp}\pi_{\varepsilon},\mu_2)+\int_{\Omega^2}c_{\varepsilon}\, d\pi_{\varepsilon}+\varepsilon^{3d+2}\#\supp(\pr_2)_{\sharp}\pi_{\varepsilon} \label{ineq: masterineq} \\
&\leq \frac{1}{\varepsilon}W_1((\pr_2)_{\sharp}\pi'_j,\mu_2)+\int_{\Omega^2}c_{\varepsilon}\, d\pi'_j+\varepsilon^{3d+2}\#\supp(\pr_2)_{\sharp}\pi'_j. \nonumber
\end{align}
Since the map $p_j$ is a nearest point projection to $\mathcal{N}_j$, we have
\[
W_1((\pr_2)_{\sharp}\pi'_j, \mu_2)=W_1((p_j)_{\sharp}\mu_2,\mu_2) \leq \int_{\Omega}\rho(x,p_j(x))\, d\mu_2(x)\leq \frac1j
\]
and
\[
\#\supp(\pr_2)_{\sharp}\pi'_j=\supp(p_j)_{\sharp}\mu_2\leq \#\cl{N}_j,
\]
so we obtain
\[
C_{\varepsilon}(\pi_{\varepsilon})\leq \frac{1}{j\varepsilon}+\int_{\Omega^2}c_{\varepsilon}\, d\pi'_j + \varepsilon^{3d+2}\#\cl{N}_j.
\]
In particular, this yields that
\[
\frac{1}{\varepsilon}W_1((\pr_2)_{\sharp}\pi_{\varepsilon},\mu_2)\leq C_{\varepsilon}(\pi_{\varepsilon}) \leq \frac1{j\varepsilon} + \int_{\Omega^2}c_{\varepsilon}\, d\pi'_j +
\varepsilon^{3d+2}\#\cl{N}_j.
\]
By multiplying $\varepsilon>0$ and letting $\varepsilon\to 0$, we have
\[
\limsup_{\varepsilon\to 0}W_1((\pr_2)_{\sharp}\pi_{\varepsilon},\mu_2) \leq \frac1j.
\]
Since $j$ is arbitrary, $(\pr_2)_{\sharp}\pi_{\varepsilon}$ converges to $\mu_2$ weakly as 
$\varepsilon\to 0$.
By letting $j_{\varepsilon}\approx \varepsilon^{-2}$, (\ref{ineq: masterineq}) yields
\[
\int_{\Omega^2}\rho\, d\pi_{\varepsilon} 
\leq \int_{\Omega^2}c_{\varepsilon}\, d\pi_{\varepsilon}
\leq \int_{\Omega^2}\rho\, d\pi'_{j_{\varepsilon}}+O(\varepsilon).
\]
Letting $\varepsilon\to 0$ yields
\[
\int_{\Omega^2}\rho\, d\pi_* \leq \liminf_{\varepsilon\to 0}\int_{\Omega^2}\rho\, d\pi_{\varepsilon}\leq W_1(\mu_1,\mu_2).
\]
Since $\pi_*\in\Pi(\mu_1,\mu_2)$, we have $\pi_*\in \cl{O}_0(\mu_1,\mu_2)$. Moreover, by using 
\[
\int_{\Omega^2}\rho\, d\pi_{\varepsilon} \geq W_1(\mu_1,(\pr_2)_{\sharp}\pi_{\varepsilon})\geq W_1(\mu_1,\mu_2)-W_1(\mu_2,(\pr_2)_{\sharp}\pi_{\varepsilon}),
\]
\begin{align*}
\int_{\Omega^2}\rho(x,y)\, d\pi'_j 
&= \int_{\Omega^2}\rho(x,p_j(y))\, d\pi \\
&\leq \int_{\Omega^2}\rho(x,y)\, d\pi +\int_{\Omega^2}\rho(y,p_j(y)) \leq W_1(\mu_1,\mu_2)+\frac1j
\end{align*}
and (\ref{ineq: masterineq}), we get
\[
\frac{1-\varepsilon}{\varepsilon}W_1(\mu_2,(\pr_2)_{\sharp}\pi_{\varepsilon})+\varepsilon\int_{\Omega^2}\lvert x-y\rvert^2\, d\pi_{\varepsilon}
\leq \frac{1+\varepsilon}{j\varepsilon}+\varepsilon\int_{\Omega^2}\lvert x-y\rvert^2\, d\pi'_j +O(\varepsilon^{3d+2}).
\]
In particular, by letting $\varepsilon<1$ and $j_{\varepsilon}\approx\varepsilon^{-3}$, we obtain
\[
\int_{\Omega^2}\lvert x-y\rvert^2\, d\pi_{\varepsilon} \leq \int_{\Omega^2}\lvert x-y\rvert^2\, d\pi'_{j_{\varepsilon}} + O(\varepsilon^2),
\]
which leads us to 
\[
\int_{\Omega^2}\lvert x-y\rvert^2\, d\pi_* \leq \int_{\Omega^2}\lvert x-y\rvert^2\, d\pi.
\]
Since $\pi\in\cl{O}_1(\mu_1,\mu_2)$, we have $\pi_*\in\cl{O}_1(\mu_1,\mu_2)$.
\end{proof}

We consider restricted measures and interpolated measures.
For a Borel measure $\pi$ on $\Omega^2$ and a Borel set $B\subset \Omega^2$ of positive measure,
we denote by $\pi\lfloor_B$ the restriction of $\pi$ to $B$.
We interpolate measures along line segments in $\Omega$. For $0\leq t\leq 1$, we put
\[
P_t(x,y):=(1-t)x+ty,\, x,y\in\Omega.
\]
\begin{lem}\label{lem: densitybound}
Let $B\subset \Omega^2$ be a Borel set, let $\varepsilon>0$ and $\pi_{\varepsilon} \in D_{\varepsilon}$.
Denote by $\mu_1^{\varepsilon ,B}$ and $\mu_2^{\varepsilon,B}$ the first and the second marginal of $\pi_{\varepsilon}\lfloor_B$ respectively.
Then, $\mu_1^{\varepsilon, B}$ is absolutely continuous with respect to $\cl{L}^n$ and we have
the following {\rm (1), (2)} and {\rm (3)}.
\begin{enumerate}
\item The measure $\pi_{\varepsilon}\lfloor_B$ is a $c_{\varepsilon}$-optimal transport plan from $\mu_1^{\varepsilon, B}$ to
$\mu_2^{\varepsilon, B}$.
\item For any $t\in [0,1)$, the interpolated measure $(P_t)_{\sharp}(\pi_{\varepsilon}\lfloor_B)$ is absolutely continuous with respect to $\cl{L}^n$.
\item If the density $\rho_{\varepsilon, B}$ of $\mu_1^{\varepsilon,B}$ is essentially bounded, then
so is the density of $(P_t)_{\sharp}(\pi_{\varepsilon}\lfloor_{B})$. In this case, we have
\[
\left\|\frac{d(P_t)_{\sharp}(\pi_{\varepsilon}\lfloor_B)}{d\cl{L}^n}\right\|_{L^\infty} \leq (1-t)^{-n}\|\rho_{\varepsilon,B}\|_{L^\infty}.
\]
\end{enumerate}
\end{lem}

\begin{proof}
The absolute continuity of $\mu_1^{\varepsilon,B}$ is clear.
The assertion (i) follows from the stability of optimality under the restriction of measures.
Let us prove (ii). If $t=0$, then the assertion is trivial. We may assume $t>0$.
Let $\{y_i\}_{i\in I}$ be the support of $\mu_2^{\varepsilon,B}$ and put $\Omega_i:=\supp(\pi_{\varepsilon}\lfloor_{\Omega\times\{y_i\}})$ and $\Omega_i(t):= P_t(\Omega\times \{y_i\})$. By a simple calculation, for any Borel set $A\subset \Omega$,
\begin{align}
(P_t)_{\sharp}(\pi_{\varepsilon}\lfloor_B)(A) \nonumber
&= \pi_{\varepsilon}\lfloor_B(P_t^{-1}(A)) \nonumber \\
&\leq \sum_{i\in I} \pi_{\varepsilon}\lfloor_B\left(P_t^{-1}(A\cap \Omega_i(t))\right) \nonumber \\
&= \sum_{i\in I} \pi_{\varepsilon}\lfloor_B\left(\frac{1}{1-t}(A\cap \Omega_i(t)-ty_i)\times \Omega\right) \nonumber \\
&= \sum_{i\in I}  \mu_1^{\varepsilon,B}\left(\frac{1}{1-t}(A\cap \Omega_i(t)-ty_i)\right) \nonumber \\
&\leq (1-t)^{-n}\sum_{i\in I} \|\rho_{\varepsilon,B}\|_{L^{\infty}}\mathcal{L}^n(A\cap\Omega_i(t)) \label{ineq: densitybound},
\end{align}
where we have used the translation invariance of $\cl{L}^n$ in the last inequality.
We now claim that the family $\{\Omega_i(t)\}_{i\in I}$ is pairwise disjoint for any $0<t<1$.
Assume the contrary.
Then, there exist two points $x_i,x_j\in \Omega$ such that
\[
(1-t)x_i+ty_i=(1-t)x_j+y_j.
\]
From this, we have
\begin{align*}
x_i-y_j
&= (1-t)x_j-(1-t)y_j+t(x_i-y_i) \\
&= (1-t)(x_j-y_j)+t(x_i-y_i)
\end{align*}
and
\[
c_{\varepsilon}(x_i,y_j) < (1-t)c_{\varepsilon}(x_j,y_j)+tc_{\varepsilon}(x_i,y_i)
\]
by the strict convexity of $c_{\varepsilon}$. Exchanging $i$ and $j$, we also have
\[
c_{\varepsilon}(x_j,y_i) < (1-t)c_{\varepsilon}(x_i,y_i)+tc_{\varepsilon}(x_j,y_j).
\]
On the other hand, the $c_{\varepsilon}$-cyclical monotonicity of $\supp\pi_{\varepsilon}$ yields
\[
c_{\varepsilon}(x_i,y_i)+c_{\varepsilon}(x_j,y_j)
\leq
c_{\varepsilon}(x_i,y_j)+c_{\varepsilon}(x_j,y_i).
\]
This is a contradiction and the assertion (iii) follows.
\end{proof}

\begin{lem}\label{lem: restriction}
Let $\pi_k,\pi\in \mathcal{P}(\Omega^2)$, $k=1,2,\ldots$, be Borel probability measures with a common first marginal.
If $\{\pi_k\}_{k=1}^{\infty}$ converges to $\pi$ weakly, then for any Borel set $G\subset \Omega$, the sequence $\{\pi_k\lfloor_{G\times \Omega}\}_{k=1}^{\infty}$ converges to $\pi\lfloor_{G\times \Omega}$ weakly.
\end{lem}


\begin{defn}
Let $\Gamma\subset \Omega^2$ be a subset.
We put 
\[
\mathbb{T}(\Gamma):=\left\{P_t(x,y)\setrel t\in [0,1], (x,y)\in \Gamma\right\}
\]
and call it \emph{the transport set associated to $\Gamma$}.
\end{defn}
Note that if $\Gamma$ is $\sigma$-compact, then so is $\mathbb{T}(\Gamma)$.
We also observe that if a measure $\lambda\in\mathcal{P}(\Omega^2)$ is concentrated on $\Gamma$, then
the measure $(P_t)_{\sharp}\lambda$ is concentrated on $\mathbb{T}(\Gamma)$ for any $t\in [0,1]$.
Now we recall the notion of density-regular points and regular points, introduced in \cite{CP11}.

\begin{defn}\label{def: densityregularity}
Let $\pi\in\cl{P}(\Omega^2)$ be a probability measure concentrated on a $\sigma$-compact set $\Gamma\subset \Omega^2$. 
Assume that the first marginal $\mu_1$ of $\pi$ is absolutely continuous with respect to $\cl{L}^n$ and let $f$ be the density.
A point $(x,y)\in \Gamma$ is said to be \emph{density-regular} if  for any $r>0$, there exists a point $\tilde{y}\in\Omega$ and a positive number $\tilde{r}>0$ such that 
\begin{itemize}
\item $y\in B(\tilde{y},\tilde{r})\subset B(y,r)$, 
\item The point $x$ is a Lebesgue point $f$ and $\tilde{f}$. 
\item $f(x)<+\infty$, $\tilde{f}(x)>0$,
\end{itemize}
where $\tilde{f}$ is the density of $(\pr_1)_{\sharp}(\pi\lfloor_{\Omega\times B(\tilde{y},\tilde{r})})$. We denote by $\dr(\Gamma)$ the set of all density-regular points of $\Gamma$.
\end{defn}

\begin{defn}\label{def: densityregularity}
Let $\Gamma\subset \Omega^2$ be a $\sigma$-compact set.
A point $(x,y)\in \Gamma$ is said to be \emph{regular} if for any $r>0$, $x$ is a Lebesgue point of $\Gamma^{-1}(B(y,r))$,
where the set $\Gamma^{-1}(B(y,r))$ is defined by
\[
\Gamma^{-1}(B(y,r)):=\left\{x \in \Omega \setrel (x,z)\in\Gamma\ \text{for some}\ z\in B(y,r)\right\}.
\]
\end{defn}

\begin{remark}
Any density-regular point is regular.
\end{remark}

\begin{lem}\label{lem: densityregularity}
Under the same settings as in Definition \ref{def: densityregularity},
$\pi$ is concentrated on $\dr(\Gamma)$.
\end{lem}

\begin{prop}\label{prop: density regularity}
Let $\mu_1, \mu_2\in\mathcal{P}_c(\Omega)$ and let $\pi\in\mathcal{O}_2(\mu_1,\mu_2)$.
Assume that the measure $\mu_1$ is absolutely continuous with respect to $\cl{L}^n$ and let $f$ be the density of 
$\mu_1$.
Then, for any point $(x,y)\in \dr(\Gamma)$ with $x\neq y$ and positive real number $r>0$, 
\[
\liminf_{\delta\to0}\frac{\mathcal{L}^n\left(\mathbb{T}\left(\Gamma\cap(B(x,{\delta/4})\times B(y,r))\right)\cap B(x,{\delta})\right)}{\mathcal{L}^n\left(B(x,{\delta})\right)}>0.
\]
\end{prop}

\begin{proof}
Define a Borel set $F\subset \Omega$ by
\[
F:=\left\{z\in \Omega \setrel {\tilde{f}(x)}/{2}\leq \tilde{f}(z)\right\} \cap 
\left\{z\in \Omega \setrel f(z)\leq f(x)+1\right\}.
\]
Since $x$ is a Lebesgue point of $f$ and $\tilde{f}$, we have
\[
\lim_{s\to 0}\frac{\mathcal{L}^n\left(F\cap B(x,s)\right)}{\mathcal{L}^n\left(B(x,s)\right)}=1.
\]
Fix a positive real number $\delta>0$ so that $\delta/2(\lvert x-y\rvert + r)<1$ and
\begin{equation}\label{ineq: lebesguediffthm}
\frac{\mathcal{L}^n\left(F\cap B(x,s)\right)}{\mathcal{L}^n\left(B(x,s)\right)}\geq \frac12
\end{equation}
hold for any $0<s<\delta$.
By putting $F_{\delta}:=B(x,{\delta/4})\cap F$, we have
\begin{equation}\label{ineq: diff}
\frac{\tilde{f}(x)}{4}\mathcal{L}^n\left(B(x,{\delta/4})\right)
\leq \frac{\tilde{f}(x)}{2}\mathcal{L}^n(F_{\delta}).
\end{equation}
Put $A_{\delta}:=F_{\delta}\times B(\tilde{y},{\tilde{r}})$ and let $f_{\delta}$ be the density of $(\pr_1)_{\sharp}\left(\pi\lfloor_{A_{\delta}}\right)$. By the definition of $\tilde{f}$ and $F_{\delta}$, we have
\begin{align*}
f_{\delta}(z) &= \limsup_{s\to 0}\frac{\pi\left((B(z,s)\cap F_{\delta})\times B(\tilde{y},{\tilde{r}})\right)}{\cl{L}^n\left(B(z,s)\right)} \\
&= \limsup_{s\to 0}\frac{1}{\cl{L}^n\left(B(z,s)\right)}\int_{B(z,s)\cap F_{\delta}} \tilde{f}\, d\cl{L}^n \\
&\geq \frac{\tilde{f}(x)}{2}\limsup_{s\to 0}\frac{\cl{L}^n\left(B(z,s)\cap F_{\delta}\right)}{\cl{L}^n\left(B(z,s)\right)}.
\end{align*}
Lebesgue's differentiation theorem tells
\[
f_{\delta}(z)\geq \frac{\tilde{f}(x)}{2}
\]
for $\mathcal{L}^n$-almost all $z\in F_{\delta}$.
Integrating this on $F_{\delta}$, we obtain
\begin{equation}\label{ineq: integrate}
\frac{\tilde{f}(x)}{2}\cl{L}^n(F_{\delta})\leq \int_{F_{\delta}}f_{\delta}(z)\, dz =(\pr_1)_{\sharp}(\pi\lfloor_{A_{\delta}})(F_{\delta}).
\end{equation}
We also observe that $P_t(z,w)\in B(x,\delta)$ for any $0<t<\delta/(2(\lvert x-y\rvert+r))$, $z\in B(x,\delta/4)$ and $w\in B(y,r)$, which yields
\begin{equation}\label{ineq: ball}
(\pr_1)_{\sharp}(\pi\lfloor_{A_{\delta}})(F_{\delta}) \leq (P_t)_{\sharp}(\pi\lfloor_{A_{\delta}})(B(x,{\delta})).
\end{equation}
Combining (\ref{ineq: diff}), (\ref{ineq: integrate}) and (\ref{ineq: ball}), we reach
\begin{equation}\label{ineq: interpolate}
\frac{\tilde{f}(x)}{4}\mathcal{L}^n\left(B(x,{\delta/4})\right)
\leq (P_t)_{\sharp}(\pi\lfloor_{A_{\delta}})(B(x,{\delta})).
\end{equation}

Since $\pi\in\mathcal{O}_2(\mu_1,\mu_2)$, there exist a sequence $\{\varepsilon_k\}_{k=1}^{\infty}$ of positive real numbers converging to $0$ and a sequence $\{\pi_{k}\in D_{\varepsilon_k}\}_{k=1}^{\infty}$ of  Borel probability measures converging to $\pi$ weakly.
From Lemma \ref{lem: densitybound} (iii), for any natural number $k$, we have
\begin{align*}
\|(P_t)_{\sharp}\left(\pi_{k}\lfloor_{F_{\delta}\times \Omega}\right)\|_{L^{\infty}} &\leq (1-t)^{-n}\|(\pr_1)_{\sharp}\left(\pi_k\lfloor_{F_{\delta}\times \Omega}\right)\|_{L^{\infty}} \\
&\leq (1-t)^{-n}\|f\lfloor_{F_{\delta}}\|_{L^{\infty}} \leq 2^{n}(f(x)+1).
\end{align*}
By Lemma \ref{lem: restriction}, we know that the sequence $\{\pi_{k}\lfloor_{F_{\delta}\times \Omega}\}_{k=1}^{\infty}$
 converges to $\pi\lfloor_{F_{\delta}\times \Omega}$ weakly. Since the $L^{\infty}$-bound is stable under the weak convergence, we obtain
\[
\|(P_t)_{\sharp}\left(\pi\lfloor_{A_{\delta}}\right)\|_{L^{\infty}}
\leq \|(P_t)_{\sharp}(\pi\lfloor_{F_{\delta}\times \Omega})\|_{L^{\infty}}
\leq 2^n(f(x)+1).
\]
The fact that
$(P_t)_{\sharp}(\pi\lfloor_{A_{\delta}})$ is concentrated on
$\mathbb{T}\left(\Gamma \cap (B(x,{\delta/4})\times B(y,r))\right)$ implies that
\begin{align*}
&(P_t)_{\sharp}(\pi\lfloor_{A_{\delta}})(B(x,{\delta})) \\
&= (P_t)_{\sharp}(\pi\lfloor_{A_{\delta}})\left(\mathbb{T}\left(\Gamma \cap (B(x,{\delta/4})\times B(y,r))\right) \cap B(x,{\delta})\right) \\
 &\leq 2^n(f(x)+1)\cl{L}^n\left(\mathbb{T}\left(\Gamma \cap (B(x,{\delta/4})\times B(y,r))\right) \cap B(x,{\delta})\right).
\end{align*}
Combining this with (\ref{ineq: interpolate}), we arrive at
\[
\frac{\cl{L}^n\left(\mathbb{T}\left(\Gamma\cap(B(x,{\delta/4})\times B(y,r))\right)\cap B(x,{\delta})\right)}
{\cl{L}^n\left(B(x,{\delta/4})\right)}\geq \frac{\tilde{f}(x)}{2^{n+2}(f(x)+1)},
\]
which is equivalent to 
\[
\frac{\cl{L}^n\left(\mathbb{T}\left(\Gamma\cap(B(x,{\delta/4})\times B(y,r))\right)\cap B(x,{\delta})\right)}
{\cl{L}^n\left(B(x,{\delta})\right)}\geq \frac{\tilde{f}(x)}{2^{3n+2}(f(x)+1)}.
\]
This completes the proof.
\end{proof}
The following theorem is the main result of this section and Theorem \ref{thm: main theorem} follows from this.
\begin{thm}\label{main theorem}
Let $\mu_1,\mu_2\in \mathcal{P}_c(\Omega)$.
If the measure $\mu_1$ is absolutely continuous with respect to $\mathcal{L}^n$, then $\mathcal{O}_2(\mu_1,\mu_2)$ consists of a single element $\pi_0$ and the transport plan $\pi_0$ is induced by a map. 
\end{thm}

\begin{proof}
Let $\pi\in\mathcal{O}_2(\mu_1,\mu_2)$. By virtue of Lemma \ref{lem: map} and Lemma \ref{lem: densityregularity}, it suffices to prove that the set $\dr(\Gamma)$ is concentrated on the graph of a map.
Assume that there exist two points $(x_0,y_0), (x_0,y_1)\in \dr(\Gamma)$ with $y_0\neq y_1$.
Then we have either $(y_1-y_0)\cdot (y_0-x_0)<0$ or $(y_0-y_1)\cdot (y_1-x_0)<0$. By symmetry, we may assume 
$(y_1-y_0)\cdot (y_0-x_0)<0$. By continuity, for sufficiently small $r>0$, any points $x\in B(x_0,r)$, $y'\in B(y_0,r)$ and $y\in B(y_1,r)$ satisfy
\[
(y-y')\cdot (y'-x)<0.
\]
On the other hand, by the $\Gamma$-regularity of $(x_0,y_1)\in \Gamma$, we have
\[
\lim_{\delta\to 0}\frac{\cl{L}^n\left(\pr_1(\Gamma \cap (\Omega \times B(y_1,r)))\cap B(x_0,{\delta})\right)}{\cl{L}^n(B(x_0,{\delta}))}=1
\]
and by the density-regularity of $(x_0,y_0)\in \dr(\Gamma)$, we have
\[
\liminf_{\delta\to 0}\frac{\cl{L}^n\left(\mathbb{T}(\Gamma\cap (B(x_0,{\delta/4}))\times B(y_0,r))\cap B(x_0,{\delta})\right)}{\cl{L}^n(B(x_0,{\delta}))}>0.
\]
Thus for sufficiently small $0<\delta<r$, there exists a point $x\in B(x_0,{\delta})$ such that
\[
x\in \pr_1(\Gamma \cap (\Omega \times B(y_1,r)))\cap \mathbb{T}(\Gamma\cap (B(x_0,{\delta/4}))\times B(y_0,r))).
\]
This implies the existence of points $y\in B(y_1,r)$, $x'\in B(x_0,{\delta/4})$ and $y'\in B(y_0,r)$ such that $(x,y), (x',y')\in \Gamma$ and $x\in [x',y']$. From the restricted monotonicity, we have
\[
(y-y')\cdot (x-x')\geq 0.
\]
Since $x$ lies in the line segment $[x',y']$, we have
\[
x-x'=a(y'-x')
\]
for some $a>0$. Thus we obtain
\[
(y-y')\cdot (x-x') = a(y-y')\cdot (y'-x)<0,
\]
which is a contradiction.
\end{proof}

\section{Hilbert geometries}
In this section, we recall the notion of Hilbert geometries, introduced by Hilbert \cite{Hilbert1895} himself.
We refer to \cite{Papa14} for details.

Let $\Omega\subset \mathbb{R}^n$ be a bounded convex open domain.
For any distinct two points $x,y\in \Omega$,
the \emph{Hilbert metric} $h_{\Omega}(x,y)$ from $x$ to $y$ is given by
\[
h_{\Omega}(x,y):=\frac{1}{2}\log\frac{\lvert y-x'\rvert\lvert x-y'\rvert}{\lvert x-x'\rvert\lvert y-y'\rvert},
\]
where the points $x'=x+s(y-x), s<0$ and $y'=x+t(y-x), t>0$ are the intersection of the affine straight line passing through $x$ and $y$ and the boundary $\partial \Omega$ (see Figure 1).
The function $h_{\Omega}$ defines a complete metric on $\Omega$ and the metric space $(\Omega, h_{\Omega})$ is called the \emph{Hilbert geometry} for $\Omega$.
The Hilbert geometry for a unit open disk coincides with the Cayley-Klein model of the hyperbolic geometry.
If the boundary $\partial \Omega$ is $C^2$-regular and strictly convex, then the Hilbert metric on $\Omega$ is induced by the smooth Finsler structure $F$, given by 
\[
F(x,v)=\frac{\lvert v\rvert}{2}\left(\frac{1}{\lvert x-a\rvert}+\frac{1}{\lvert x-b\rvert}\right),
\]
where the points $a=x+sv, s<0,$ and $b=x+tv, t>0$, are the intersection of the affine straight line passing through $x$ with the direction $v$ (see Figure 1). In this case, the Finsler manifold $(\Omega, F)$ has the constant flag curvature $-1$ (cf. \cite{Okada83}).
\begin{figure}[htbp]
\centering
\includegraphics[width=8cm]{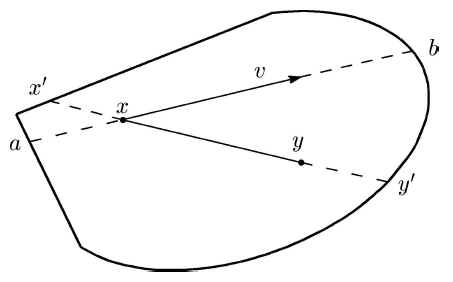}
\caption{Hilbert metric}
\end{figure}
For any bounded convex open set $\Omega$, the Hilbert metric $h_{\Omega}$ satisfies (i) and (ii) of Assumption A in  Section 1. 
To verify (iii), we use the following result, proved by Ohta \cite{Ohta13}.
\begin{thm}[\cite{Ohta13}]
Let $\Omega\subset \mathbb{R}^n$ be a bounded convex open domain and $h_{\Omega}$ the Hilbert metric.
Then the metric-measure space $(\Omega,h_{\Omega},\cl{L}^n)$ satisfies the curvature-dimension condition $\cd(K,N)$ for
\[
K=-(n-1)-\frac{(n+1)^2}{N-n}\ \text{and}\ N>n.
\]
\end{thm}
The curvature-dimension condition $\cd(K,N)$, which we do not explain in this paper, is a generalization of the condition that the weighted Ricci curvature is bounded below by $K$ and the dimension is bounded above by $N$.
If a metric-measure space $(X,d,\mathfrak{m})$ satisfies the curvature-dimension condition $\cd(K,N)$ for $K\in \mathbb{R}$ and finite $N>1$, then the base measure $\mathfrak{m}$ satisfies the local doubling property.
In particular, the Hilbert metrics satisfy (iii) of Assumption A, so we obtain Corollary \ref{cor: Hilb geom}.

\bibliographystyle{plain}

\end{document}